\documentclass[10pt,reqno]{amsart}
\usepackage{amsmath}
\usepackage{amscd,amsthm,amsfonts,amsopn,amssymb,mathrsfs}

\newtheorem{theorem}{Theorem}[section]
\newtheorem{proposition}[theorem]{Proposition}

\newtheorem{lemma}[theorem]{Lemma}
\theoremstyle{remark}

\DeclareMathOperator{\supp}{supp\,}

\allowdisplaybreaks

\begin{document}

\title{An optimal decay estimate for the linearized water wave equation in 2D}
\author{Aynur Bulut}
\address{Department of Mathematics, University of Michigan, 530 Church Street, Ann Arbor, MI 48109-1043}
\email{abulut@umich.edu}
\maketitle

\begin{abstract}
We obtain a decay estimate for solutions to the linear dispersive equation $iu_t-(-\Delta)^{1/4}u=0$ for $(t,x)\in\mathbb{R}\times\mathbb{R}$.  This corresponds to a factorization of the linearized water wave equation $u_{tt}+(-\Delta)^{1/2}u=0$.  In particular, by making use of the Littlewood-Paley decomposition and stationary phase estimates, we obtain decay of order $|t|^{-1/2}$ for solutions corresponding to data $u(0)=\varphi$, assuming only bounds on $\lVert \varphi\rVert_{H_x^1(\mathbb{R})}$ and $\lVert x\partial_x\varphi\rVert_{L_x^2(\mathbb{R})}$.  As another application of these ideas, we give an extension to equations of the form $iu_t-(-\Delta)^{\alpha/2}u=0$ for a wider range of $\alpha$.
\end{abstract}

\section{Introduction}

In this note we establish decay properties for solutions to the initial value problem
\begin{align}
\left\lbrace\begin{array}{rl}iu_t(t,x)-(-\Delta)^{1/4}u(t,x)&=0,\quad (t,x)\in \mathbb{R}\times\mathbb{R},\\
u(0)&=\varphi\in \mathcal{S}(\mathbb{R}).\end{array}\right.\label{eq1}
\end{align}
where we use the notation $-\Delta=-\partial_{xx}$.

This equation arises as a factorized form of the linearized 2D water wave equation,
\begin{align}
u_{tt}+(-\Delta)^{1/2}u&=0.\label{eq2}
\end{align}
In particular, the operator $\partial_{tt}+(-\Delta)^{1/2}$ can be written as $-(i\partial_t+(-\Delta)^{1/4})(i\partial_t-(-\Delta)^{1/4})$.  Recall that the water wave system is a system of quasilinear PDEs corresponding to a free boundary problem modelling the motion of a liquid (assumed to be irrotational and inviscid) subject to gravity and sitting below a region of air; in this model, the influence of surface tension is not treated (although there has been recent progress in studying models which incorporate this effect --- see, for instance, the works \cite{CHS,ABZ,IP-CAP}).

There has been much recent progress in the study of the nonlinear water wave system in two and higher dimensions, beginning with the almost global result of Wu in \cite{W} (see also the global results for the 3D system in \cite{W2} and \cite{GMS}).  Very recently, there have been a number of results giving global existence of solutions in 2D; see the recent works of Ionescu--Pusateri \cite{IP2}, Alazard--Delort \cite{AD,AD2}, Hunter--Ifrim-Tataru \cite{JIT}, and Ifrim--Tataru \cite{IT}.  As is often the case in analyzing long-time behavior for equations having a dispersive nature, a key role in the analysis contained in each of these works is played by decay estimates.  In the context of the approach of \cite{W}, we note that while such estimates are often obtained as a consequence of an $L^1$--$L^\infty$ dispersive estimate, this is not compatible with the bootstrap procedure used to solve the nonlinear equation, which is based on $L^2$-based norms.  The decay used in the work \cite{W} is therefore obtained as the result of an application of the Klainerman vector field method.  This allows one to obtain $|t|^{-1/2}$ decay (which matches the decay rate for smooth solutions obtained from $L^1$--$L^\infty$ estimates); however, it requires the inclusion of certain vector field-based norms in the bootstrap procedure, which lead to some additional assumptions on the initial data (see \cite{JB,JB-2} for further remarks on this point).

The main result of this note is the following theorem, which gives $|t|^{-1/2}$ decay, assuming only control over $\lVert \varphi\rVert_{H_x^1(\mathbb{R})}$ and $\lVert x\partial_x \varphi\rVert_{L_x^2(\mathbb{R})}$.

\begin{theorem}
\label{thm1}
Let $\Phi:\mathbb{R}\rightarrow\mathbb{R}$ be given by $\Phi(\xi)=|\xi|^{1/2}$ for $\xi\in\mathbb{R}$.  Then there exists $C>0$ such that estimate
\begin{align}
\lVert e^{it\Phi(D)}\varphi\rVert_{L_x^{\infty}}\leq C(1+|t|)^{-1/2}\Big(\lVert \varphi\rVert_{H_x^1}+\lVert x\partial_x\varphi\rVert_{L_x^2}\Big)\label{goal}
\end{align}
holds for every $t\in\mathbb{R}$ and $\varphi\in \mathcal{S}(\mathbb{R})$, where the operator $e^{it\Phi(D)}$ is defined by
\begin{align*}
\widehat{e^{it\Phi(D)}\varphi}=e^{it\Phi(\xi)}\widehat{\varphi}\quad \textrm{for}\quad \varphi\in \mathcal{S}(\mathbb{R}).
\end{align*}
\end{theorem}

We also refer to the recent work of Beichman \cite{JB-2} (see also \cite{JB}), where an interesting class of decay estimates with a lower order rate of decay for (\ref{eq1}) and (\ref{eq2}) is obtained; see Proposition 3.4 and Theorem 4.7 of \cite{JB-2}.  The arguments in \cite{JB-2} are based on a reduced form of the Klainerman vector field method, which avoids the need to invoke the additional vector field present in \cite{W}.

The proof of Theorem $\ref{thm1}$ is given in Section $2$.  In Section $3$, we describe how the estimates used in this argument can be applied to establish a class of similar results for certain classes of operators in the scale $i\partial_t-(-\Delta)^{\alpha/2}$, $\alpha\in\mathbb{R}$, $\alpha\neq 1$.

\vspace{0.1in}

{\noindent \bf Acknowledgments}: The author is grateful to Sijue Wu for helpful discussions, and to the anonymous referee for careful reading and for very valuable comments on the manuscript.  The author was partially supported by NSF Grant DMS-1361838.

\section{Proof of the main result}
\label{sec-ip-est}

In this section, we give the proof of Theorem $\ref{thm1}$, the decay estimate for $e^{it\Phi(D)}\varphi$.  The proof of Theorem $\ref{thm1}$ is closely related to a linear estimate obtained by Ionescu and Pusateri in their study of the global existence problem for the full water wave equation \cite{IP1,IP2} (see, for instance, \cite[Lemma 2.3]{IP1}), where the relevant tools are the Littlewood-Paley decomposition and the method of stationary phase.  

We begin by specifying some brief notational conventions.  For $f\in \mathcal{S}(\mathbb{R})$, we will use $\widehat{f}$ to denote its Fourier transform (in all of our calculations we will slightly abuse notation in the interest of brevity by omitting all factors of $2\pi$).  The notation $A\lesssim B$ shall mean, as usual, that there exists a constant $C>0$ such that $A\leq CB$ holds.  Moreover, we shall use the abbreviations $\partial_x=\frac{d}{dx}$ and $\partial_\xi=\frac{d}{d\xi}$.

We use the usual Littlewood-Paley projection operators $P_k$, $k\in\mathbb{Z}$, acting on functions $f:\mathbb{R}\rightarrow\mathbb{C}$.  In particular, let $\psi\in C_c^\infty(\mathbb{R})$ be given such that $0\leq \psi(x)\leq 1$ for $x\in\mathbb{R}$, $\supp \psi\subset (-2,2)$, $\psi(x)=1$ for $|x|\leq 1$, and $\psi(x)=\psi(-x)$ for $x\in\mathbb{R}$.  For each $k\geq 1$, one then defines $\psi_k:\mathbb{R}\rightarrow \mathbb{R}$ by $\psi_k(x)=\psi(x/2^k)-\psi(x/2^{k-1})$.  The operators $P_k$, $k\in\mathbb{Z}$, are then defined via 
\begin{align*}
\widehat{P_kf}(\xi)=\psi_k(\xi)\widehat{f}(\xi).
\end{align*}

We recall the relevant Bernstein inequalities for these operators (that is, inequalities exploiting frequency localization to insert/remove derivatives and to move between $L^p$ norms).  In particular, for every $1\leq p,q\leq\infty$ with $p\leq q$, we have
\begin{align}
\lVert P_kg\rVert_{L_x^q(\mathbb{R})}\lesssim 2^{k(\frac{1}{p}-\frac{1}{q})}\lVert P_kg\rVert_{L_x^p(\mathbb{R})}\label{bern}
\end{align}
and
\begin{align}
\lVert P_kg\rVert_{L_x^p(\mathbb{R})}\sim 2^{-sk}\lVert (-\Delta)^{s/2}P_kg\rVert_{L_x^p(\mathbb{R})}\label{bern2}
\end{align}
for all $g\in \mathcal{S}(\mathbb{R})$.

We now state two technical lemmas which give estimates for $\widehat{P_k\varphi}$ in terms of the norms appearing on the right side of the desired decay estimate ($\ref{goal}$).
\begin{lemma}
\label{lem1}
Let the operators $P_k$, $k\in\mathbb{Z}$, be as defined in the beginning of Section $2$.  We then have
\begin{align}
\lVert \partial_\xi\widehat{P_k\varphi}\rVert_{L_\xi^2}&\lesssim 2^{-k}(\lVert \varphi\rVert_{L_x^2}+\lVert x\partial_x\varphi\rVert_{L_x^2}).\label{eq-weight-bernstein}
\end{align}
for all $k\in\mathbb{Z}$ and $\varphi\in \mathcal{S}(\mathbb{R})$.
\end{lemma}
\begin{proof}
Fix $k\in\mathbb{Z}$, and write
\begin{align*}
\nonumber 2^k\lVert \partial_\xi\widehat{P_k\varphi}\rVert_{L_\xi^2}&\leq 2^k\lVert (\partial_\xi \psi_k)(\xi)\widehat{\varphi}(\xi)\rVert_{L_\xi^2}+2^k\lVert \psi_k(\xi)\partial_\xi\widehat{\varphi}\rVert_{L_\xi^2}.
\end{align*}
Recalling that $\partial_\xi\psi_k(\xi)=2^{-k}\psi'(\xi/2^k)-2^{-(k-1)}\psi'(\xi/2^{k-1})\lesssim 2^{-k}$, and noting that $|\xi|\sim 2^k$ for $\xi\in \supp \psi_k$, this quantity is then bounded by a multiple of
\begin{align*}
\lVert \widehat{\varphi}\rVert_{L_\xi^2}+\lVert \xi\partial_\xi\widehat{\varphi}\rVert_{L_\xi^2}&=\lVert \varphi\rVert_{L_x^2}+\Big\lVert \partial_x\Big[ ix\varphi(x)\Big]\Big\rVert_{L_x^2}
\end{align*}
where we have used the Plancherel identity to obtain the equality.  The desired inequality ($\ref{eq-weight-bernstein}$) now follows immediately.
\end{proof}

\begin{lemma}
\label{lem2}
For every $s\in \mathbb{R}$ with $\frac{1}{2}<s<1$ there exists $C=C(s)>0$ such that
\begin{align}
\lVert \widehat{P_k\varphi}\rVert_{L_\xi^\infty}&\leq C\left(\lVert P_k\varphi\rVert_{L_x^2}+2^{-sk}(\lVert \varphi\rVert_{L_x^2}+\lVert x\partial_x\varphi\rVert_{L_x^2})\right)\label{eq-further-claim}
\end{align}
for all $k\in \mathbb{Z}$ and $\varphi\in \mathcal{S}(\mathbb{R})$.
\end{lemma}

\begin{proof}
Fix $s>\frac{1}{2}$, and use the (one-dimensional) Sobolev inequality followed by interpolation to obtain
\begin{align}
\lVert \widehat{P_k\varphi}\rVert_{L_\xi^\infty}&\lesssim \lVert \widehat{P_k\varphi}\rVert_{H_\xi^s}\lesssim \lVert P_k\varphi\rVert_{L_x^2}+\lVert P_k\varphi\rVert_{L_x^2}^{1-s}\lVert \partial_\xi\widehat{P_k\varphi}\rVert_{L_\xi^2}^s.\label{eq-further-0}
\end{align}

We now observe that, by the Plancherel identity followed by H\"older's inequality (recalling that $\widehat{P_k\varphi}$ is supported on a set of measure $O(2^k)$) and the Sobolev embedding,
\begin{align}
\lVert P_k\varphi\rVert_{L_x^2}&=\lVert \widehat{P_k\varphi}\rVert_{L_\xi^2}\lesssim 2^{k/4}\lVert \widehat{P_k\varphi}\rVert_{L_\xi^4}\lesssim 2^{k/4}\lVert (-\Delta_{\xi})^{1/8}\widehat{P_k\varphi}\rVert_{L_\xi^2}.\label{eq-further-1}
\end{align}
Interpolation then shows that the right side of ($\ref{eq-further-1}$) is bounded by
\begin{align*}
2^{k/4}\lVert P_k\varphi\rVert_{L_x^2}^{3/4}\lVert (-\Delta_\xi)^{1/2}\widehat{P_k\varphi}\rVert_{L_\xi^2}^{1/4}\lesssim 2^{k/4}\lVert P_k\varphi\rVert_{L_x^2}^{3/4}\lVert \partial_\xi\widehat{P_k\varphi}\rVert_{L_\xi^2}^{1/4},
\end{align*}
so that we obtain
\begin{align*}
\lVert P_k\varphi\rVert_{L_x^2}&\lesssim 2^k\lVert \partial_\xi\widehat{P_k\varphi}\rVert_{L_x^2}.
\end{align*}

We now use this inequality to estimate the right side of ($\ref{eq-further-0}$).  In particular, this gives the bound
\begin{align*}
\lVert\widehat{P_k\varphi}\rVert_{L_\xi^\infty}&\lesssim \lVert P_k\varphi\rVert_{L_x^2}+2^{k(1-s)}\lVert \partial_\xi\widehat{P_k\varphi}\rVert_{L_\xi^2}.
\end{align*}
Application of Lemma $\ref{lem1}$ now gives ($\ref{eq-further-claim}$) as desired.
\end{proof}

Having obtained Lemma $\ref{lem1}$ and Lemma $\ref{lem2}$, we now give the proof of Theorem $\ref{thm1}$.

\begin{proof}[Proof of Theorem $\ref{thm1}$]
We argue as in \cite{IP1}.  Let $\varphi\in\mathcal{S}(\mathbb{R})$ and $t\in \mathbb{R}$ be given.  Since the embedding $H_x^1(\mathbb{R})\hookrightarrow L_x^\infty(\mathbb{R})$ implies $\lVert e^{it\Phi(D)}\varphi\rVert_{L_x^\infty}\lesssim \lVert \varphi\rVert_{H_x^1(\mathbb{R})}$, the desired inequality is immediate for all $|t|\leq 1$.  We therefore assume $|t|\geq 1$ and fix an arbitrary point $x\in\mathbb{R}$.  Then, invoking the Littlewood-Paley decomposition (with notation as in the beginning of Section $2$), we bound $|e^{it\Phi(D)}\varphi(x)|$ by
\begin{align*}
&\sum_{\substack{k\in\mathbb{Z}\\2^k\leq \lambda(t)}} \Big| (P_ke^{it\Phi(D)}\varphi)(x)\Big|+\sum_{\lambda(t)\leq 2^k\leq \Lambda(t)} \Big| (P_ke^{it\Phi(D)}\varphi)(x)\Big|\\
&\hspace{0.4in}+\sum_{2^k\geq \Lambda(t)} \Big| (P_ke^{it\Phi(D)}\varphi)(x)\Big|=:(A)+(B)+(C),
\end{align*}
where we have set $\lambda(t)=2^{10}(1+|t|)^{-1}$ and $\Lambda(t)=2^{-10}(1+|t|)$, 

The terms (A) and (C) are estimated identically as in \cite{IP1}.  In particular, the Bernstein inequality ($\ref{bern}$) gives
\begin{align*}
&(A)\leq \sum_{2^k\leq \lambda(t)} \lVert P_ke^{it\Phi(D)}\varphi\rVert_{L_x^\infty}\lesssim \sum_{2^k\leq \lambda(t)} 2^{k/2}\lVert P_ke^{it\Phi(D)}\varphi\rVert_{L_x^2}\\
&\hspace{0.4in}\lesssim\lambda(t)^{1/2}\lVert \varphi\rVert_{L_x^2}\lesssim (1+|t|)^{-1/2}\lVert \varphi\rVert_{L_x^2}.
\end{align*}

Similarly, application of ($\ref{bern}$) and ($\ref{bern2}$) gives
\begin{align*}
(C)&\leq \sum_{2^k\geq \Lambda(t)} 2^{k/2}\lVert P_k\varphi\rVert_{L_x^2}\lesssim \sum_{2^k\geq \Lambda(t)} 2^{-k/2}\lVert \varphi\rVert_{H_x^1}\lesssim (1+|t|)^{-1/2}\lVert \varphi\rVert_{H_x^1}.
\end{align*}

For the estimate of (B), we further decompose the sum into three parts, corresponding to summations over
\begin{align*}
I_1:=\{k:\lambda(t)\leq 2^k\leq \Lambda(t)\,\, \textrm{and}\,\, 2^{k/2}\leq |t/x|/16\},
\end{align*}
\begin{align*}
I_2:=\{k:\lambda(t)\leq 2^k\leq \Lambda(t)\,\, \textrm{and}\,\, |t/x|/16\leq 2^{k/2}\leq 16|t/x|\},
\end{align*}
and
\begin{align*}
I_3:=\{k:\lambda(t)\leq 2^k\leq \Lambda(t)\,\, \textrm{and}\,\, 2^{k/2}\geq 16|t/x|\}.
\end{align*}

The contributions of $I_1$ and $I_3$ to $(B)$ are again estimated identically as in \cite{IP1}, leading to the bounds
\begin{align}
(B)_{j}\lesssim |t|^{-1/2}(\lVert \varphi\rVert_{L_x^2}+\lVert x\partial_x\varphi\rVert_{L_x^2}),\quad j=1,3,\label{eq-i1-0}
\end{align}
where
\begin{align}
(B)_{j}:=\sum_{k\in I_j}\bigg|\int_{\mathbb{R}} e^{i(x\xi+t\Phi(\xi))}\widehat{P_k\varphi}(\xi)d\xi\bigg|,\quad j=1,2,3,\label{eq-i1}
\end{align}
are the contributions of $I_1$, $I_2$, and $I_3$ to $(B)$.  In order to give a complete presentation we give the argument for $j=1$; the estimate for $j=3$ follows from identical considerations.  Note that integration by parts gives
\begin{align}
(B)_{1}&\leq \sum_{k\in I_1}\bigg(\frac{1}{|t|}\int_{\mathbb{R}} \frac{|\partial_\xi\widehat{P_k\varphi}(\xi)|}{|\frac{x}{t}+\Phi'(\xi)|} d\xi + \int_{\mathbb{R}}\frac{|t\Phi''(\xi)|}{|x+t\Phi'(\xi)|^2}|\widehat{P_k\varphi}(\xi)| d\xi\bigg)\label{eq-b1}
\end{align}

We fix $k\in I_1$ and estimate each term in the above sum.  Note that, since $k\in I_1$ implies $|x/t|\leq 2^{-k/2}/16$, the identity $|\Phi'(\xi)|=\frac{1}{2}|\xi|^{-1/2}$ gives 
\begin{align}
|\frac{x}{t}+\Phi'(\xi)|\geq |\Phi'(\xi)|-\Big|\frac{x}{t}\Big|&\geq \frac{1}{2}|\xi|^{-1/2}-\frac{2^{-k/2}}{16},\label{eq-rhs1}
\end{align}
and it follows that for all $\xi\in \supp \psi_k$ (which corresponds to $2^{k}\leq |\xi|\leq 2^{k+1}$), one has
\begin{align}
(\ref{eq-rhs1})\geq \frac{1}{2}\bigg(\frac{2^{-k/2}}{\sqrt{2}}\bigg)-\frac{2^{-k/2}}{16}\gtrsim 2^{-k/2}.\label{eq-rhs2}
\end{align}
Applying Cauchy-Schwarz and noting that the measure of $\supp \psi_k$ is bounded by a multiple of $2^k$, the first term in the parentheses in ($\ref{eq-b1}$) is bounded by a multiple of 
\begin{align*}
|t|^{-1}2^{k/2}\lVert \partial_\xi\widehat{P_k\varphi}(\xi)\rVert_{L_\xi^1}\lesssim |t|^{-1}2^{k}\lVert \partial_\xi\widehat{P_k\varphi}\rVert_{L_\xi^2}.
\end{align*}
Applying Lemma $\ref{lem1}$, this is bounded by $|t|^{-1}(\lVert \varphi\rVert_{L_x^2}+\lVert x\partial_x \varphi\rVert_{L_x^2})$.

We next estimate the second term in the parentheses in ($\ref{eq-b1}$), for which we invoke a similar argument.  Recall that $k\in I_1$ is fixed.  Using ($\ref{eq-rhs1}$)--($\ref{eq-rhs2}$) once again, we obtain
\begin{align}
\nonumber \int_{\mathbb{R}} \frac{|t||\Phi''(\xi)|}{|x+t\Phi'(\xi)|^2}|\widehat{P_k\varphi}(\xi)|d\xi&=\frac{1}{|t|}\int_{\mathbb{R}} \frac{|\xi|^{-3/2}}{|\frac{x}{t}+\Phi'(\xi)|^2}|\widehat{P_k\varphi}(\xi)|d\xi\\
&\leq \frac{2^k}{|t|}\int_{\mathbb{R}} |\xi|^{-3/2}|\widehat{P_k\varphi}(\xi)|d\xi,\label{eq-term2}
\end{align}
where we have used the computation $|\Phi''(\xi)|=\frac{1}{4}|\xi|^{-3/2}$.  Recalling again that $\xi\in \supp \psi_k$ implies $|\xi|\geq 2^k$, and invoking Cauchy-Schwarz as before, we obtain the bound
\begin{align*}
(\ref{eq-term2})&\lesssim \frac{1}{2^{k/2}|t|}\lVert \widehat{P_k\varphi}\rVert_{L_x^1}\lesssim |t|^{-1}\lVert P_k\varphi\rVert_{L_x^2}\lesssim |t|^{-1}\lVert \varphi\rVert_{L_x^2}
\end{align*}

To complete the estimate, we must now take the sum over $k\in I_1$.  In particular, since $I_1$ is contained in $\{k:\lambda(t)\leq 2^k\leq \Lambda(t)\}$, this set has at most $O(\log|t|)$ elements, and we obtain
\begin{align*}
(B)_{1}&\lesssim |t|^{-1}\sum_{k\in I_1}(\lVert \varphi\rVert_{L_x^2}+\lVert x\partial_x\varphi\rVert_{L_x^2})\\
&\lesssim |t|^{-1}\log|t|(\lVert \varphi\rVert_{L_x^2}+\lVert x\partial_x\varphi\rVert_{L_x^2})\\
&\lesssim |t|^{-1/2}(\lVert \varphi\rVert_{L_x^2}+\lVert x\partial_x\varphi\rVert_{L_x^2}),
\end{align*}
as desired.

It remains to estimate the contribution of $I_2$ to $(B)$.  As is done in \cite{IP1}, we apply the method of stationary phase.  In particular, for each $k\in I_2$, set $Q_{t,x}(\xi)=x\xi+t\Phi(\xi)$ and write
\begin{align}
\bigg|\int_{\mathbb{R}} e^{iQ_{t,x}(\xi)}\widehat{P_k\varphi}(\xi)d\xi\bigg|&=\bigg|\int_{|\xi|\in [2^{k},2^{k+1}]} e^{iQ_{t,x}(\xi)}\psi_k(\xi)\widehat{\varphi}(\xi)d\xi\bigg|.\label{eq-i2-1}
\end{align}
It follows from explicit computation of $\Phi'$ that there exists a unique $\xi_0\in \mathbb{R}$ with $Q_{t,x}'(\xi_0)=0$ (in fact, one has $|\xi_0|=\frac{1}{4}|\frac{t}{x}|^2$).

Fix a parameter $\ell_0\in\mathbb{Z}$ to be chosen later in the argument.  Then,
\begin{align}
\nonumber (\ref{eq-i2-1})&\leq \bigg|\int_{\{\xi:2^{k}\leq |\xi|\leq 2^{k+1}\}} e^{iQ_{t,x}(\xi)}\psi_k(\xi)\widehat{\varphi}(\xi)\psi(\tfrac{\xi-\xi_0}{2^{\ell_0}})d\xi\bigg|\\
&\hspace{0.2in}+\sum_{\ell=\ell_0+1}^\infty \bigg|\int_{\{\xi:2^{k}\leq |\xi|\leq 2^{k+1}\}} e^{iQ_{t,x}(\xi)}\psi_k(\xi)\widehat{\varphi}(\xi)\psi_\ell(\xi-\xi_0)d\xi\bigg|.\label{eq-i2-1b}
\end{align}

Fix $s>\frac{1}{2}$.  Then Lemma $\ref{lem2}$ implies that the first term in ($\ref{eq-i2-1b}$) is bounded by
\begin{align}
\lVert \psi_k\widehat{\varphi}\rVert_{L_\xi^\infty}\lVert \psi\Big(\tfrac{\xi-\xi_0}{2^{\ell_0}}\Big)\rVert_{L_\xi^1}&\lesssim 2^{\ell_0}\lVert P_k\varphi\rVert_{L_x^2}+2^{\ell_0-sk}(\lVert \varphi\rVert_{L_x^2}+\lVert x\partial_x\varphi\rVert_{L_x^2}).\label{eq-firstterm}
\end{align}

We now turn our attention to the subsequent terms.  Let $\ell\geq \ell_0+1$ be given.  Then, integrating by parts,
\begin{align}
\nonumber &\bigg|\int e^{iQ_{t,x}(\xi)}\psi_k(\xi)\widehat{\varphi}(\xi)\psi_{\ell}(\xi-\xi_0)d\xi\bigg|\\
\nonumber &\hspace{0.2in}\leq \bigg\lVert \frac{(\partial_\xi \widehat{P_k\varphi})(\xi)\psi_{\ell}(\xi-\xi_0)}{|Q'_{t,x}(\xi)|}\bigg\rVert_{L_\xi^1}+\bigg\lVert \frac{(\widehat{P_k\varphi})(\xi)\partial_{\xi}[\psi_\ell(\xi-\xi_0)]}{|Q'_{t,x}(\xi)|}\bigg\rVert_{L_\xi^1}\\
&\hspace{0.4in}+\bigg\lVert \frac{(\widehat{P_k\varphi})(\xi)(\psi_{\ell})(\xi-\xi_0)}{|Q'_{t,x}(\xi)|^2}Q''_{t,x}(\xi)\bigg\rVert_{L_\xi^1}\label{eq-i2-2}
\end{align}

Define $q_0(k,\ell):=\inf\{|Q'_{t,x}(\xi)|:\xi\in \supp \psi_k, \xi-\xi_0\in \supp \psi_\ell\}$.  Let $\xi\in \supp \psi_k$ be given satisfying $\xi-\xi_0\in\supp \psi_\ell$.  We then have $|\xi|\sim 2^k$, $|\xi-\xi_0|\sim 2^\ell$, and thus
\begin{align}
|t|^{-1}2^{(3k/2)-\ell}|Q'_{t,x}(\xi)|\sim \frac{1}{|\xi-\xi_0|}\left|\frac{x}{t}|\xi|^{3/2}+\frac{\xi}{2}\right|\label{identity}
\end{align}
on this set.  The right side of this expression is bounded from below by
\begin{align*}
\frac{1}{|\xi-\xi_0|}\left(\left|\frac{x}{t}\right||\xi|^{3/2}-\frac{1}{2}|\xi|\right)&\geq \frac{|\xi|}{|\xi|+|\xi_0|}\left(\frac{16}{2^{k/2}}|\xi|^{1/2}-\frac{1}{2}\right)\gtrsim \frac{|\xi|}{|\xi|+|\xi_0|}
\end{align*}
where in the first inequality we have recalled that we are working under the condition $k\in I_2$, and in the second inequality we have used $|\xi|\sim 2^k$.  Noting that under the current hypotheses on $\xi$ and $\xi_0$, 
\begin{align*}
\frac{|\xi|}{|\xi|+|\xi_0|}\gtrsim \frac{2^k}{2^k+\frac{1}{4}|t/x|^2}\gtrsim \frac{2^k}{2^k+\frac{(16)^2}{4}2^k},
\end{align*}
we conclude that $(\ref{identity})\gtrsim 1$, and thus $$q_0(k,\ell)\gtrsim |t|2^{\ell-\frac{3k}{2}}.$$

Combining this estimate with Cauchy-Schwarz and explicit calculation of $|\Phi''(\xi)|$, we obtain the bound
\begin{align*}
(\ref{eq-i2-2})&\lesssim \frac{2^{3k/2}}{|t|2^{\ell}}\bigg(\lVert \partial_\xi \widehat{P_k\varphi}(\xi)\rVert_{L_{\xi}^2}\lVert \psi_\ell\rVert_{L_\xi^2}+\frac{1}{2^\ell}\lVert \widehat{P_k\varphi}\rVert_{L_\xi^\infty}\lVert \psi'(\tfrac{\xi-\xi_0}{2^\ell})-\psi'(\tfrac{\xi-\xi_0}{2^{\ell-1}})\rVert_{L_\xi^1}\\
&\hspace{1.2in}+\frac{1}{2^\ell}\lVert \widehat{P_k\varphi}\rVert_{L_\xi^\infty}\lVert \psi_\ell\rVert_{L_\xi^1}\bigg).
\end{align*}
A change of variables in the $L^1$ norms shows that this expression is bounded by
\begin{align*}
\frac{2^{3k/2}}{|t|2^{\ell}}\bigg(\lVert \partial_\xi \widehat{P_k\varphi}(\xi)\rVert_{L_{\xi}^2}\lVert \psi_\ell\rVert_{L_\xi^2}+\lVert \widehat{P_k\varphi}\rVert_{L_\xi^\infty}\bigg).
\end{align*}

We now apply Lemma $\ref{lem1}$ and Lemma $\ref{lem2}$.  Application of these results gives
\begin{align*}
&\frac{2^{k/2}}{|t|2^{\ell/2}}\Big(\lVert \varphi\rVert_{L_x^2}+\lVert x\partial_x\varphi\rVert_{L_x^2}\Big)\\
&\hspace{0.4in}+ \frac{2^{3k/2}}{|t|2^{\ell}}\bigg(\lVert P_k\varphi\rVert_{L_x^2}+2^{-sk}\Big(\lVert \varphi\rVert_{L_x^2}+\lVert x\partial_x\varphi\rVert_{L_x^2}\Big)\bigg).
\end{align*}
Taking the summation in $\ell$ and recalling the estimate ($\ref{eq-firstterm}$) for the first term in ($\ref{eq-i2-1b}$), we obtain
\begin{align*}
(\ref{eq-i2-1})&\lesssim \Big(2^{\ell_0}+\frac{2^{3k/2}}{|t|2^{\ell_0}}\Big)\lVert P_k\varphi\rVert_{L_x^2}+\Big(2^{\ell_0-sk}+\frac{2^{k/2}}{|t|2^{\ell_0/2}}+\frac{2^{k(\frac{3}{2}-s)}}{|t|2^{\ell_0}}\Big)\Big(\lVert \varphi\rVert_{L_x^2}+\lVert x\partial_x\varphi\rVert_{L_x^2}\Big).
\end{align*}
Setting $s=\frac{3}{4}$ and choosing $\ell_0$ to satisfy $2^{\ell_0}\sim 2^{3k/4}/|t|^{1/2}$, this becomes
\begin{align*}
(\ref{eq-i2-1})&\lesssim \frac{2^{3k/4}}{|t|^{1/2}}\lVert P_k\varphi\rVert_{L_x^2}+\Big(\frac{1}{|t|^{1/2}}+\frac{2^{k/8}}{|t|^{3/4}}\Big)\Big(\lVert \varphi\rVert_{L_x^2}+\lVert x\partial_x\varphi\rVert_{L_x^2}\Big)\\
&\lesssim |t|^{-1/2}\lVert \varphi\rVert_{H_x^{3/4}}+(|t|^{-1/2}+2^{k/8}|t|^{-3/4})(\lVert \varphi\rVert_{L_x^2}+\lVert x\partial_x \varphi\rVert_{L_x^2})
\end{align*}

Recalling that $k\in I_2$ implies $2^k\leq \Lambda(t)\lesssim |t|$, and taking the summation over all $k\in I_2$, we obtain (since the number of elements of $I_2$ is bounded by an absolute constant; this can be seen by recalling that $k\in I_2$ implies $-8+2\log(|t/x|)\leq k\leq 8+2\log(|t/x|)$)
\begin{align}
(B)_{2}&\lesssim (|t|^{-1/2}+|t|^{-5/8})(\lVert \varphi\rVert_{H_x^{3/4}}+\lVert x\partial_x\varphi\rVert_{L_x^2}).\label{eq-b2}
\end{align}

Combining the estimates for $(A)$, $(C)$, and $(B)_1$--$(B)_3$, we obtain
\begin{align}
|(e^{it\Phi(D)}\varphi)(x)|&\lesssim |t|^{-1/2}(\lVert \varphi\rVert_{H_x^1}+\lVert x\partial_x\varphi\rVert_{L_x^2})\label{eq-claim2}
\end{align}
for $|t|\geq 1$.  We have therefore established (\ref{goal}) as desired.
\end{proof}

\section{Concluding remarks}

The argument described in this paper is also applicable to a wider range of operators in the scale $i\partial_t-(-\Delta)^{\alpha/2}$, $\alpha\in\mathbb{R}\setminus \{1\}$ (for which the treatment above corresponds to $\alpha=1/2$).  

We give an example of such an  argument in the range $\frac{1}{3}<\alpha<\frac{1}{2}$.  Indeed, in this range, after fixing $M\geq 1$ sufficiently large and replacing the partition of indices $I_1$--$I_3$ with
\begin{align}
\nonumber I'_1&:=\{k\in\mathbb{Z}:\lambda(t)\leq 2^k\leq \Lambda(t), 2^{k(1-\alpha)}\leq M^{-1}|t/x|\},\\
\nonumber I'_2&:=\{k:\lambda(t)\leq 2^k\leq \Lambda(t), M^{-1}|t/x|\leq 2^{k(1-\alpha)}\leq M|t/x|\},\quad\textrm{and}\\
I'_3&:=\{k:\lambda(t)\leq 2^k\leq \Lambda(t), 2^{k(1-\alpha)}\geq M|t/x|\},\label{i-prime}
\end{align}
we are led to the following proposition:
\begin{proposition}
Fix $\frac{1}{3}<\alpha<\frac{1}{2}$, and let $\Phi=\Phi_\alpha:\mathbb{R}\rightarrow\mathbb{R}$ be given by $\Phi_\alpha(\xi)=|\xi|^\alpha$ for $\xi\in\mathbb{R}$.  Then there exists $C=C(\alpha)>0$ such that the estimate
\begin{align}
\lVert e^{it\Phi_\alpha(D)}\varphi\rVert_{L_x^\infty}\leq C(1+|t|)^{-1/2}(\lVert\varphi\rVert_{H_x^1}+\lVert x\partial_x\varphi\rVert_{L_x^2})\label{eq-goal-prop}
\end{align}
holds for all $\varphi\in \mathcal{S}(\mathbb{R})$.
\end{proposition}

Since the proof consists essentially of ``{\it mutatis mutandis}'' adjustments of the proof of Theorem $\ref{thm1}$ above, we will only record the final chain of estimates involved in the argument.

\begin{proof}[Sketch of proof]
The argument follows the outline used in the proof of Theorem $\ref{thm1}$; in particular, fixing $\alpha<\frac{1}{2}$, the quantities (A) and (C) are estimated identically as in that argument (using the Bernstein inequalities (\ref{bern}) and (\ref{bern2}), and it remains to estimate the quantities $(B)_1$--$(B)_3$ (with $\Phi$ replaced with $\Phi_\alpha$ and the index sets $I_1$--$I_3$ replaced with $I'_1$--$I'_3$ given in (\ref{i-prime})).  For simplicity, we retain the notation $(B)_1$--$(B)_3$ for these modified quantities.

For this, we note that estimates on the integration kernel (namely, $|\frac{x}{t}+\Phi'(\xi)|\gtrsim 2^{-k(1-\alpha)}$ for $\xi\in\supp\psi_k$ and $k\in I'_1$ or $k\in I'_3$) lead to the bounds
\begin{align*}
(B)_1&\lesssim |t|^{-1}\sum_{k\in I_1} 2^{k(\frac{1}{2}-\alpha)}(\lVert \varphi\rVert_{L_x^2}+\lVert x\partial_x\varphi\rVert_{L_x^2})\\
&\lesssim |t|^{-1}\Lambda(t)^{\frac{1}{2}-\alpha}(\lVert\varphi\rVert_{L_x^2}+\lVert x\partial_x\varphi\rVert_{L_x^2})\\
&\lesssim |t|^{-\frac{1}{2}-\alpha}(\lVert\varphi\rVert_{L_x^2}+\lVert x\partial_x\varphi\rVert_{L_x^2})
\end{align*}
and, likewise,
\begin{align*}
(B)_2&\lesssim |t|^{-\frac{1}{2}-\alpha}(\lVert \varphi\rVert_{L_x^2}+\lVert x\partial_x\varphi\rVert_{L_x^2}).
\end{align*}

It remains to estimate $(B)_2$.  We continue to emulate the procedure used in the proof of Theorem $\ref{thm1}$, for which the bound on the analogue of $q_0(k,\ell)$ is 
\begin{align*}
q_0(k,\ell)\gtrsim |t|2^{\ell-(2-\alpha)k}.
\end{align*}
This leads to the estimate of the quantity corresponding to (\ref{eq-i2-1}) by
\begin{align*}
&\bigg(2^{\ell_0}+\frac{2^{(2-\alpha)k}}{|t|2^{\ell_0}}\bigg)\lVert P_k\varphi\rVert_{L_x^2}\\
&\hspace{0.2in}+\bigg(2^{\ell_0-sk}+\frac{2^{(1-\alpha)k}}{|t|2^{\ell_0/2}}+\frac{2^{(2-\alpha-s)k}}{|t|2^{\ell_0}}\bigg)(\lVert\varphi\rVert_{L_x^2}+\lVert x\partial_x\varphi\rVert_{L_x^2}.
\end{align*}
where $\ell_0\geq 1$ and $\frac{1}{2}<s<1$ are fixed parameters.  Choosing $s=\frac{2-\alpha}{2}$ and $\ell_0$ such that $2^{\ell_0}\sim 2^{(2-\alpha)k/2}/|t|^{1/2}$, we obtain a bound of
\begin{align*}
\frac{2^{(2-\alpha)k/2}}{|t|^{1/2}}\lVert P_k\varphi\rVert_{L_x^2}+\bigg(\frac{1}{|t|^{1/2}}+\frac{2^{(2-3\alpha)k/4}}{|t|^{3/4}}\bigg)(\lVert \varphi\rVert_{L_x^2}+\lVert x\partial_x\varphi\rVert_{L_x^2}).
\end{align*}

We now conclude the estimate by again appealing to the estimate $2^k\lesssim \Lambda(t)\lesssim |t|$ and taking summation over $k\in I_2$.  This gives
\begin{align}
(B)_2&\lesssim (|t|^{-1/2}+|t|^{\frac{2-3\alpha}{4}-\frac{3}{4}})(\lVert \varphi\rVert_{H_x^{(2-\alpha)/2}}+\lVert x\partial_x\rVert_{L_x^2}).\label{eq-b2-2}
\end{align}
Since the $H_x^{(2-\alpha)/2}$ norm is bounded by the $H_x^1$ norm, (\ref{eq-b2-2}) implies the desired estimate (\ref{eq-goal-prop}) provided that $\alpha$ satisfies $\frac{2-3\alpha}{4}-\frac{3}{4}<-\frac{1}{2}$; that is, when $\alpha>\frac{1}{3}$.  This completes the proof of the proposition.
\end{proof}

\end{document}